\documentclass[11pt,fleqn]{amsart}

\usepackage[paper=a4paper]{geometry}

\usepackage{geometry}
\usepackage[foot]{amsaddr}
\usepackage{amssymb}
\usepackage{enumerate,multicol}
\usepackage{mathtools}

\usepackage{hyperref}

\usepackage{color}

\usepackage[utf8,latin1]{inputenc}
\usepackage[spanish,english]{babel}
\usepackage{mathpazo}
\usepackage{euler}
\usepackage[alphabetic,initials]{amsrefs}
\usepackage{amsfonts,amssymb,amsmath}
\usepackage{dsfont}
\usepackage{mathtools}
\usepackage{graphicx}
\usepackage[poly,arrow,curve,matrix]{xy}
\usepackage{wrapfig}

\newtheorem{theorem}{Theorem}[section]
\newtheorem{proposition}[theorem]{Proposition}
\newtheorem{corollary}[theorem]{Corollary}
\newtheorem{lemma}[theorem]{Lemma}
\theoremstyle{definition}

\theoremstyle{remark}
\newtheorem{remark}{Remark}[theorem]

\newcommand\NN{\mathbb N}

\newcommand\ZZ{\mathbb Z}

\newcommand\ot{\otimes}
\renewcommand\to{\longrightarrow}
\renewcommand\phi{\varphi}
\newcommand\eps{\varepsilon}

\newcommand\R{\mathcal R}

\DeclareMathOperator\Id{\mathsf{Id}}

\renewcommand\k{\Bbbk}

\DeclareMathOperator\Mod{\mathsf{Mod}}
\DeclareMathOperator\Hom{\mathsf{Hom}}
\DeclareMathOperator\Ext{\mathsf{Ext}}

\DeclareMathOperator\id{id}

\DeclareMathOperator\Add{Add}

\begin{document}

\title{Behaviour of injective dimension with respect to regradings}
\author{Andrea Solotar and Pablo Zadunaisky}
\address{IMAS and Dto. de Matem\'{a}tica, Facultad de Ciencias Exactas y Naturales,
	Universidad de Buenos Aires, Ciudad Universitaria, Pabell\'{o}n 1, 1428) Buenos Aires,
	Argentina}
\email{asolotar@dm.uba.ar, pzadub@dm.uba.ar}
\thanks{This work has been supported by the projects  UBACYT 20020130100533BA, PIP-CONICET
	112--201101--00617, PICT 2011--1510 and MATHAMSUD-GR2HOPF. The first author is a
	research member of CONICET (Argentina).}
\date{}

\begin{abstract}
Given a left noetherian $\k$-algebra $A$ graded by a group $G$, an injective object $I$ in
the category of $G$-graded $A$-modules and a morphism from $G$ to another group $G'$, we
provide bounds for the injective dimension of $I$ as a $G'$-graded $A$-module. For this,
we use three {\em change of grading} functors. Most of the constructions concerning these
functors work in the context of $H$-comodule algebras, where $H$ is a Hopf algebra, so we
develop them in this general context.
\end{abstract}

\maketitle



\section{ Introduction}

In \cite{vdB}, M. Van den Bergh states that if $A$ is a left noetherian $\NN$-graded
connected algebra over a field $\k$ and $I$ is an injective object in the category of
$\ZZ$-graded $A$-modules, then $I$ has injective dimension at most one when considered as an
$A$-module. He leaves the proof of this fact as ``a pleasant exercise in homological
algebra''. 
In this note we prove a general version of this result: 
we  provide bounds for the injective dimension
of a graded injective module when the grading changes.
This is useful for
example when one wants to study how $A$-modules behave with respect to a property which
is valid for graded $A$-modules. 
Notice that B. Fossum and H. Foxby had
already proven  in \cite{FF}*{Theorem 4.10} the statement in Van den Bergh'a article for
$A$ commutative, using localization techniques. Also, A. Yekutieli gave in \cite{Yek} a
detailed proof of Van den Bergh's statement. Our proof is completely unrelated to the ones
found in either of these references and it deals with arbitrary grading groups and change
of gradings in the non necessarily commutative case.

Just to get an idea of the situation, let us start by looking at a possible solution to
Van den Bergh's exercise. Let $A$ be a $\ZZ$-graded algebra and let $N$ be a left
$A$-module. We can turn $N \ot_{\k } \k[t,t^{-1}]$ into a
$\ZZ$-graded left $A$-module, with the action of an element $a \in A$ of degree $l \in
\ZZ$ given by $a \cdot (n \ot t^r) = an \ot t^{r+l}$ for all $n \in N$ and $r \in \ZZ$.
There is an $A$-linear surjective map $N \ot_{\k } \k[t,t^{-1}] \to N$, induced by the
projection $\k[t,t^{-1}] \to \k[t,t^{-1}]/(t-1)\cong \k$, with kernel $N \ot_{\k } (t-1)
\k[t,t^{-1}]$. Thus we obtain an exact sequence of $A$-modules
\begin{align*} 
	\xymatrix{ 
		0 \ar[r] & N \ot_{\k } \k[t,t^{-1}] \ar[r]^{\cdot(t-1)} & N \ot_{\k } \k[t,t^{-1}] \ar[r] &N
		\ar[r] &0.  
	}	 
\end{align*}
As a sequence of $\k$-modules, it is the tensor product of $N$ with the minimal projective
resolution of $\k \cong\k[t,t^{-1}]/(t-1)$ as a $\k[t,t^{-1}]$-module.

Let $I$ be an injective object in the category of $\ZZ$-graded left $A$-modules. Since 
$N\ot_{\k } \k[t,t^{-1}]$ is a $\ZZ$-graded left $A$-module, it is natural to ask whether the fact
that $I$ is \emph{graded} injective implies that $\Ext_A^i(N \ot_{\k } \k[t,t^{-1}], I) = 0$ for
all $i >0$; in Prop. \ref{caso-vdb},
we prove that this holds if $A$ is noetherian. Thus, in this case we have a resolution of
length $1$ of $N$ by left $A$-modules which are acyclic with respect to the functor
$\Hom_A(-,I)$; this proves that $\Ext_A^2(N,I) = 0$. Since $N$ is arbitrary, we deduce
that the injective dimension of $I$ in the category of $A$-modules is at most $1$.

Motivated at first by the example of change of grading considered in \cite{RZ}, the
objective of this note is to put this result in a more general perspective, showing how
injective dimension changes when we change the grading group over a fixed algebra. We
show that the general case follows the same pattern as the $\ZZ$-graded case, and requires
little more than general homological algebra. For this, we define three change of grading
functors. These functors arise in the more general situation of $A$-modules endowed with
two different comodule structures over two Hopf algebras related by a morphism.

We prove the following result.

{\bf Theorem:}  Let $A$ be a noetherian $G$-graded $\k$-algebra and let $\phi:G\to G'$ be
a morphism of groups. Let $n$ be the projective dimension of $\k G'$ as a $G'$-graded left
$\k G$-module. Given a $G$-graded injective $A$-module $I$, the injective dimension of $I$
when considered as a $G'$-graded $A$-module through $\phi$ is less than or equal to $n$.

In Section \ref{principal} we first recall some definitions and fix notations.  Afterwards
we define the change of grading functors in the Hopf algebra setting and prove useful
properties about them.

In Section \ref{groups} we specialize to the group-graded
situation and we prove our main result.

In Section \ref{bounds} we obtain bounds for the injective dimension of a graded module
when the grading changes through a group morphism.

Throughout this work $\k$ denotes a commutative ring with unit and $A$ is a $\k$-algebra,
projective as $\k$-module. All unadorned tensor products are over $\k$, and all Hopf
algebras will also be projective as $\k$-modules. 

We thank Mariano Su\'arez-Alvarez for a careful reading of a previous version of this article.
\section{ Graded algebras and modules} \label{principal}

\subsection{ Some definitions and notations} Let $A$ be a $\k$-algebra and let $G$ be a
group. The algebra $A$ is $G$-{\em graded} if it can be decomposed as the direct sum of
sub-$\k$-modules $A_g$ indexed by $G$ such that $A_g A_{g'} \subseteq A_{gg'}$. A $G$-{\em
graded module} $M$ over $A$ is an $A$-module with a decomposition as a direct sum of
sub-$\k$-modules $M_g$ such that $A_g M_{g'} \subseteq M_{gg'}$. The $\k$-module $M_g$ is
called the \emph{homogeneous component} of $M$ of degree $g$. We consider $\k$ to be a
$G$-graded algebra with $\k_{1_G} =\k$ and all other components equal to $0$.

An $A$-linear morphism $f:N \to M$ between $G$-graded left $A$-modules is said to be
homogeneous of degree $g$ if $f(N_{g'}) \subseteq M_{g'g}$ for all $g'$ in $G$. 
A $G$-grading on $A$ is equivalent to a right $\k G$-comodule algebra structure on $A$, and a
$G$-graded left $A$-module is a left $A$-module with a compatible structure of $\k
G$-comodule, that is, the following diagram commutes:

\[\xymatrixcolsep{5pc}\xymatrix{
A\otimes M \ar[r]^{\mu_M} \ar[d]^{\rho_{A\otimes M}}& M \ar[d]^{\rho_M}\\ 
A\otimes M\otimes kG \ar[r]^{\mu_m\otimes Id} & M\otimes \k G}
\]

where $\mu_M$ denotes the action of $A$ on $M$ and $\rho_M$, $\rho_{a\otimes M}$ denote the respective $\k G$ coactions.
Given a Hopf algebra $H$, we will denote by ${}_A\Mod^H$ the category of
left $A$-modules with compatible structure of right $H$-comodules. 

The morphisms in ${}_A\Mod^{\k G}$, which we shall denote $\Hom^{\k
G}_{A-}(-, -)$ are homogeneous of trivial degree. For any $f \in \Hom^{\k G}_{A-}(N,M)$
and $g\in G$ we write $f_g$ for the $\k$-linear map $f_g: N_g \to M_g$ obtained by
restriction and corestriction of $f$, and we call it the \emph{homogeneous component of $f$ in
degree $g$}. Notice that $f$ is determined by its homogeneous components.

Given a $G$-graded left $A$-module $M$ and $g\in G$, we denote by $M(g)$ the $A$-module
$M$ with a new $G$-grading, whose homogeneous components are given by $M(g)_{g'} =
M_{g'g}$. If $f\in \Hom^{\k G}_{A-}(M,M')$, we define $f(g): M(g) \to
M'(g)$ to be the morphism whose underlying function is the same as that of $f$. It is easy
to see that $f(g)$ is homogeneous of trivial degree. Thus we obtain the
\emph{$g$-shift functor}, denoted by $\Sigma_g: {}_AMod^{\k G} \to {}_AMod^{\k G}$.
One can check that $\Sigma_g \circ \Sigma_{g^{-1}} = \Id$, so the $g$-shift functor
is an automorphism of ${}_AMod^{\k G}$.

An obvious example of a $G$-graded algebra is the group algebra $\k G$. We will write
$u_g$ for the canonical generator of the homogeneous component of $\k G$ of degree $g$,
that is $(\k G)_g = \k u_g$.

\subsection{Change of grading functors} \label{change-of-grading-functors}

Let $H$ and $H'$ be two Hopf algebras and let $\phi: H \to H'$ be a
morphism of Hopf algebras.  Suppose we are given a right $H$-comodule algebra $A$ with
structure morphism $\rho_A^H: A \to A \ot H$. The composition $1 \ot \phi \circ \rho_A^H:
A \to A \ot H'$ turns $A$ into an $H'$-comodule algebra. Of course the same idea works
with left $H$-comodule algebras; in particular, we consider $H$ as a left $H'$-comodule
algebra with structure map $\phi \ot 1 \circ \Delta_{H}$.

The morphism $\phi$ induces two functors \begin{align*} \phi_!: {}_A\Mod^{H} &\to
{}_A\Mod^{H'} & \phi^*: {}_A\Mod^{H'} \to {}_A\Mod^{H} \end{align*} which we now define.

Given an object $M$ of ${}_A\Mod^{H}$ with structure morphism $\rho_M^{H}: M \to M \ot H$,
the object $\phi_!(M)$ coincides with $M$ as $A$-module, while its structure morphism is
$\rho_{\phi_!(M)}^{H'} = 1 \ot \phi \circ \rho_M^{H}: M \to M \ot H'$. Furthermore, given
a morphism $f: M \to M'$ in ${}_A\Mod^{H}$, we set $\phi_!(f) = f$. It is easy to check
that the functor $\phi_!$ is well defined.  We say that $\phi_!$ is obtained by restriction
along $\phi$. 

Now, given an object $N$ of ${}_A\Mod^{H'}$, we define $\phi^*(N)$ as $N \square_{H'} H$,
where $\square_{H'}$ is the cotensor product of $H'$-comodules; here we consider $H$ as a
left $H'$-comodule through $\phi$. Notice that $N \otimes H$ has a left $A \ot H$-module
structure. Since $\rho^{H}_A: A \to A \otimes H$ is a ring homomorphism, it induces a left
$A$-module structure on $N \ot H$. The subspace $\phi^*(N) = N
\square_{H'} H \subseteq N \ot H$ is an $A$-submodule of $N \ot H$.  For any morphism $f: N
\to N'$ in ${}_A\Mod^{H'}$, we set $\phi^*(f): \phi^*(N) \to \phi^*(N')$ to be  $f \ot
\Id: N \square_{H'} H \to N' \square_{H'} H$, that is the restriction and corestriction of
$f \ot \Id$.

The cotensor product has been defined originally in \cite{EilMoo}. 
The functors $\phi_!$ and $\phi^*$ are $A$-equivariant versions of those introduced in \cite{Doi}*{1.2}.
Proposition \ref{P:adjoint} below corresponds to Proposition 6 of that article.

\begin{proposition} 
\label{P:adjoint} 
The functor $\phi^*$ is right adjoint to $\phi_!$.
\end{proposition} 
\begin{proof} 
	The image of the coaction map $\rho_M^{H}: M \to M \ot H$
	is contained in $M \square_{H'} H$, and hence induces a map $\iota_M: M \to
	\phi^*(\phi_!(M))$, which is natural in $M$ and compatible with the action of $A$.
	Also, there is a natural transformation	$\eps_N: \phi_!(\phi^*(N)) \to N$ given by
	the composition of $Id \square_{H'} \phi$ with the canonical isomorphism $N
	\square_{H'} H' \cong N$, which is also $A$-linear. These are respectively the unit and
	the counit of the adjoint pair $(\phi_!, \phi^*)$.  
\end{proof}

By definition the functor $\phi_!$ is exact and reflects exactness, meaning that a complex in
${}_A\Mod^{H}$ is exact if and only if its image by $\phi_!$ is exact. As shown in
\cite{Doi}*{Proposition 5}, the functor $\phi^*$ is exact if $H$ is an injective
$H'$-comodule, for example if it is free. By standard properties of adjoint functors, we
obtain the following corollary. For the proof of the second statement, see \cite{Wei}*{Prop. 2.3.10}.

\begin{corollary} 
The functor $\phi^*$ sends injective objects to injective objects. Furthermore, if $H$
is an injective $H'$-comodule, the functor $\phi_!$ sends projective objects to
projective objects.  
\end{corollary}

The hypothesis of the last part of the corollary is satisfied for example when $H'$ is a
group algebra over a field, since in this case $H'$ is cosemisimple and thus every
comodule is injective.

For the rest of this subsection, $\k$ will be a field.
Let $L=\{x\in H: \rho^{H'}_H(x)= x\otimes 1\}$ be the subalgebra of coinvariants of $H$ as left $H'$-comodule. Recall that $L
\subseteq H$ is called \emph{cleft} if there exists a convolution-invertible left
$H'$-comodule morphism $\gamma: H' \to H$ 
and in that case there is a left $H'$-colinear, right $L$-linear isomorphism of Hopf algebras $H\to H' {}_\sigma\# L$
and the extension $L\subseteq H$ is $H'$-Galois; see \cite{M}*{Theorem 7.2.2} for details.

By \cite{SURVEY}*{5.1}, the category ${}_A\Mod^{H'}$ is a Grothendieck category. If the extension $L\subseteq H$ is cleft, then $H$ is a free $H'$-comodule, so it
is injective and in particular $\phi^*: {}_A\Mod^{H'} \to {}_A\Mod^{H}$ is an exact
functor, hence preserves colimits. Thus, by Freyd's Adjoint Functor Theorem, $\phi^*$
has a right adjoint, which we denote by $\phi_*$.

The following corollary is also a consequence of adjointness.

\begin{corollary} 
\label{C:COG-functors-properties} 
If the extension  $L\subseteq H$ is $H'$-cleft, then the functors $\phi_!$ and $\phi^*$ send
projective objects to projective objects.
\end{corollary}

We set some more notation. Given a Hopf algebra $H$ and an object $N$ of ${}_A\Mod^{H}$ we
denote by $\Add(N)$ the full subcategory of ${}_A\Mod^{H}$ whose objects are direct
summands in ${}_A\Mod^{H}$ of  coproducts of copies of $N$. In other words $\Add(N)$ is
the smallest full subcategory of ${}_A\Mod^{H}$ containing $N$ and closed by direct sums
and direct summands.

\begin{proposition} 
\label{resolution} 
Suppose $L \subseteq H$ is cleft and $H'$ is cocomutative. 
If $N$ is an object of ${}_A\Mod^{H'}$, then there exists a resolution $S^\bullet \to N$, with $S^i$ in
$\Add (\phi_!(\phi^*(N)))$ for all $i \geq 0$. Moreover, if $n$ is the projective
dimension of $H'$ in ${}_H\Mod^{H'}$, then $S^i = 0$ for all $i > n$.
\end{proposition} 
\begin{proof} 
Since $H'$ is cocommutative, given $T \in {}_{H}\Mod^{H'}$, the $\k$-module $N	\square_{H'}
T$ is an $H'$-comodule and it has an $A$-module structure given by $a(n \ot t) =
a_0n \ot a_1t$, compatible with	the $H'$-comodule structure.	We can thus view $N \square_{H'} -$
as a functor from ${}_{H}\Mod^{H'}$ to ${}_{A}\Mod^{H'}$.

By \cite{M}*{Theorem 8.5.6} and originally \cite{Sch90}, if the antipode of $H'$ is
bijective -and this is the case here-, then there is an equivalence of categories $\Theta: {}_L\Mod \to
{}_{H}^{H'}\Mod \cong {}_{H}\Mod^{H'}$ given by $\Theta(T) = H \ot_L T$, so
${}_{H}\Mod^{H'}$ has enough projectives and each projective object is a direct summand
of a direct sum of copies of $H$.  Fix a projective resolution $P^\bullet \to H'$ in
${}_{H}\Mod^{H'}$. Since each $P^i$ is a direct summand of a free module in
${}_{H}\Mod^{H'}$ and $H$ is free as an $H'$-comodule, this is a flat resolution of $H'$
so, if we consider the
complex $N \square_{H'} P^\bullet$, it is acyclic in positive degrees and its homology in
degree zero is isomorphic to $N \square_{H'} H' \cong N$. For each $P^i$ choose $Q^i$ in
${}_{H}\Mod^{H'}$ such that $P^i \oplus Q^i$ is free of rank $n_i\in \mathbb{N}\cup \{\infty\}$, so $(N \square_{H'}
P^i) \oplus (N \square_{H'} Q^i) \cong N \square_{H'} (P^i \oplus Q^i)$, which is
isomorphic as object of ${}_{A}\Mod^{H'}$ to the direct sum of $n_i$ copies of
$\phi_!(\phi^*(N))$. Thus setting $S^\bullet = N \square_{H'} P^\bullet$ the proof is
complete.
\end{proof}

If $M\in {}_{A}\Mod^{H}$, there is an $A$-linear map $\zeta: M \ot L \to
\phi^*(\phi_!(M))$ such that $m \ot l \mapsto m_0 \ot m_1l$. This map is injective since
the corestriction of $\eta: \phi^*(\phi_!(M)) \to M \ot H$ defined as $m \ot g \mapsto m_0
\ot S(m_1)g$ is a left inverse of $\zeta$. Clearly $\zeta$ is an isomorphism if $H$ and
$M$ are finite-dimensional, and it is easily checked that this also holds if $H$ and $H'$
are group algebras. We do not know if this holds in more general situations.

\section{The group algebra case} 
\label{groups}

We now focus on the case in which $H=\k G$ and $H'=\k G'$ are
group algebras and $\phi:H \to H'$ is induced by a group morphism $\hat \phi: G \to
G'$. In this case, we have $L = \k \hat L$ with $\hat L = \ker \hat \phi$. 
As a particular case of the previous constructions, the morphism $\hat \phi: G\to G'$ induces a
$G'$-grading on $A$ via the functor $\phi_!$, the homogeneous $G'$-components are \[ A_h = \bigoplus_{g \in
\phi^{-1}(h)} A_g \] for $h \in G'$.

\begin{lemma} \label{direct-sum} Let $M$ be an object of ${}_A\Mod^{H}$. The morphism
$\bigoplus_{l \in \hat L} M[l] \to \phi^*(\phi_!(M))$ sending $m \in M[l]_g$ to $ m \ot g
\in \phi^*(\phi_!(M))$ is an isomorphism in ${}_A\Mod^{H}$, which is natural.  \end{lemma}
\begin{proof} In this case the map $\eta: \phi^*(\phi_!(M)) \to M \ot H$ defined at the end of the previous subsection 
is a two-sided inverse of $\zeta$, so $\phi^*(\phi_!(M))
	\cong M \ot L$. Now the map $\bigoplus_{l \in \hat L} M[l] \to M \ot L$ given by $m \in
	M[l] \mapsto m \ot l^{-1} \in M \ot L$ is an isomorphism in ${}_A\Mod^{H}$, and the
	morphism in the statement is the composition of this map with $\zeta$.  \end{proof}

In the case where $H=\k G$ and $H'=\k G'$ are group algebras, the functor $\phi_*$ has the following concrete description.
Explicitly, if $M$ is a $G$-graded left $A$-module and $h\in G'$,
the homogeneous component of degree $h$ of $\phi_*(M)$
is \begin{align*} \phi_*(M)_h &= \prod_{g \in \phi^{-1}(h)} M_g.  \end{align*} 

We
now define the left action of $A$ on $\phi_*(M)$ in the group algebra case; notice that it
is enough to define the action of a $G$-homogeneous element of $A$ over a $G'$-homogeneous
element of $\phi_*(M)$.  
If $a \in A_g$ and $(m_{g})_{g \in \phi^{-1}(h)} \in
\phi_*(M)_h$, then $a\cdot(m_g)_{g \in \phi^{-1}(h)} = (am_{g})_{g \in \phi^{-1}(h)}$.
Given a morphism $f$ in ${}_A\Mod^{\k G}$, the homogeneous component in degree $h$ of
$\phi_*(f)$ is given by $\phi_*(f)_h = \prod\limits_{g \in \phi^{-1}(h)}f_g$, which is
easily seen to be $A$-linear.

We refer to the functors $\phi_*, \phi^*, \phi_!$ defined in the previous paragraphs as the \emph{change of
grading} functors. As we have seen, they form an adjoint triple, i.e. $\phi^*$ is right
adjoint to $\phi_!$ and left adjoint to $\phi_*$.
These functors are not new. They appear in different guises in
\cite{Doi}, \cite{NV}, \cite{PP}, and probably many other places. However, as far as we
know this is the first time the functor $\phi_*$ is discussed in the context of graded
modules over an algebra.

\begin{remark} 
\label{remark-faltante} 
If $M$ is an object of ${}_A\Mod^{\k G}$, there is a natural isomorphism
	 \[ \phi^* \circ \phi_*(M) \cong \prod_{l \in \ker \phi} M[l]. \]
 \end{remark}

In our setting, the change of grading functors are exact, and also $\phi_*$ reflects
exactness; this does not hold for $\phi^*$ unless $\phi$ is surjective. 

Given $I, M$ of ${}_A\Mod^H$, we recall that $M$ is $I$-{\em acyclic} if $\R^i
\Hom_{A-}^H(M,I) = 0$ for all $i > 0$.

\begin{proposition} 
\label{caso-vdb} 
Assume $A$ is a noetherian $G$-graded $\k$-algebra and let $I$ be an injective object in
${}_A\Mod^{\k G}$.
\begin{enumerate} 
	\item If $M$ is an object of ${}_A\Mod^{\k G}$, then $\phi_!(M)$ is
	$\phi_!(I)$-acyclic.
	\item	Let $L=H^{coH'}$ and $\Theta: {}_L\Mod \to {}_{H}^{H'}\Mod \cong {}_{H}\Mod^{H'}$ be the equivalence of categories such that $\Theta(T) = H \ot_L T$.
	 If $P$ is a projective $L$-module and $N$ is any $G'$-graded $A$-module,
	  then $N\square_{\k G'}\Theta(P)$ is $\phi_!(I)$-acyclic.  
\end{enumerate}
\end{proposition} 
\begin{proof} 
By Remark \ref{remark-faltante} and 
the adjunctions between the change of grading functors, there are natural isomorphisms
\[
	\Hom_{A-}^{\k G'}(\phi_!(-),\phi_!(I)) \cong \Hom^{\k G}_{A-}(-,\phi^*(\phi_!(I)) \cong
	\Hom_{A-}^{\k G}(-, \bigoplus_{l \in \hat{L}} I[l]).  
\] 
Since shift functors are autoequivalences of the category ${}_AMod^{\k G}$, they preserve
injectives. Since $A$ is noetherian, the $G'$-graded $A$-module
$\phi^*(\phi_!(I))$ is injective. This follows from the graded analogue of the Bass-Papp Theorem (see for
example the proof in \cite{GW}*{Theorem 5.23}, which adapts easily to the graded case). Therefore the functor 
$\Hom_{A-}^{\k G'}(\phi_!(-),\phi_!(I))$ is exact. On the other hand, $\phi_!$
is an exact functor that sends projective objects to projective objects, so there are
natural isomorphisms $(\R^i\Hom_{A-}^{\k G'})(\phi_!(-),\phi_!(I)) \cong
\R^i(\Hom_{A-}^{\k G'}(\phi_!(-),\phi_!(I)))$, and, since the last functor is identically
zero, this proves the first statement in the proposition.
  
Now, if $P$ is a projective $L$-module, then it is a direct summand of a free $L$-module
$F$, so $N \square_{\k G'} \Theta(P)$ is a direct summand of $N \square_{\k G'}
\Theta(F)$, which in turn is isomorphic to a direct sum of copies of $\phi_!(\phi^*(N))$.
Thus $N \square_{\k G'} \Theta(F)$, being a direct sum of $\phi_!(I)$-acyclic modules, is itself
$\phi_!(I)$-acyclic and so are its direct summands. This proves item 2.
\end{proof}


We write $\Ext^{\k G,i}_{A-}(-,M)$ for the $\mathit{i}$th right derived functor of $\Hom^{\k
G}_{A-}(-,M)$.

\begin{theorem} 
\label{P:acyclic} 
Let $A$ be a noetherian $G$-graded $\k$-algebra. Let $n$	be the projective dimension of
$\k G'$ in the category ${}_{\k G}\Mod^{\k G'}$. If $I$ is an injective object in 
${}_A\Mod^{\k G}$, then the injective dimension of $\phi_!(I)$ is	less than or equal to $n$.
\end{theorem} 
\begin{proof} 
By Proposition \ref{resolution} every object $N$ of ${}_A\Mod^{\k G'}$ has a resolution
$S^\bullet$ by objects of $\Add(\phi_!(\phi^*(N)))$. Moreover, this resolution can be
chosen of length smaller than or equal to $n$. By item 2 of Proposition \ref{caso-vdb}, each object of
$\Add(\phi_!(\phi^*(N)))$ is acyclic for the functor $\Hom^{\k G'}_{A-}(-,\phi_!(I))$.
This fact implies that for all $i \geq 0$ there is an isomorphism $\Ext^{\k G',i}_{A-}(N, \phi_!(I)) \cong
H^i(\Hom_A^{\k	G}(S^\bullet, \phi_!(I)))$ for all $i \geq 0$.  In particular $\Ext^{\k
G',i}_{A-}(N,\phi_!(I)) =	0$ for any $N$ and all $i > n$, and the result follows
immediately.
\end{proof}


\section{ Injective dimension of graded modules} \label{bounds}

The results from the previous section have an easy consequence that we prove next. 

Let $\hat  L$ be, as before, the kernel of $\hat \phi : G\to G'$. Recall that the cohomological dimension of $\hat L$ over $\k$ coincides with the projective
dimension of the trivial $L$-module $\k$.  We denote by $\id^{G}_A M$ the injective
dimension of an object $M$ in ${}_A\Mod^{\k G}$. 

\begin{proposition} 
\label{C:inj-dim} 
If $M$ is an object of ${}_A\Mod^{\k G}$ and $n$ is the cohomological dimension of
$\hat L$, then:
\[ 
	\id^{G}_A M \leq \id^{G'}_A \phi_!(M) \leq \id^{G}_A M + n.  
\] 
\end{proposition} 
\begin{proof}
Since $\phi^*$ preserves injectives, it follows that $\id^{G}_A \phi^*(\phi_!(M)) \leq
\id^{G'}_A \phi_!(M)$. On the other hand,
we already know that there is a natural isomorphism $\phi^*(\phi_!(M)) \cong \bigoplus_{l
\in \hat L} M[l]$. The shift functors are autoequivalences of the
category ${}_A\Mod^{\k G}$, therefore:
\begin{align*} 
	\id^{G}_A \left( \bigoplus_{l \in \hat L} M[l] \right)  = \sup \{\id^{G}_A M[l] \mid l \in \hat L\} =
	\id^{G}_A M,
\end{align*} 
and the first inequality follows.

The second one is obvious if $\id^{G}_A M$ is infinite, so we assume it is finite and proceed by induction on $d=\id^{G}_A M$. The case $d = 0$ is Theorem
\ref{P:acyclic}. Assume  $d>0$ and let $M\to I$ be a monomorphism into an injective $G$-graded $A$-module; let
$\Omega M$ be its cokernel, which has injective
dimension $d-1$ in ${}_A\Mod^{\k G}$. We get a short exact sequence of $G'$-graded
modules 
\begin{align*} 
	0 \to \phi_!(M) \to \phi_!(I) \to \phi_!(\Omega M) \to 0.
\end{align*} 
Using a standard argument, $\id^{G'}_A \phi_!(M) \leq \max\{\id_A^{\k G}
\phi_!(I), \id^{\k G}_A \phi_!(\Omega M)\} + 1$, which by the inductive hypothesis is
smaller than or equal to $d$. This proves the second inequality.
\end{proof}

The inequalities in the statement of Corollary \ref{C:inj-dim} are sharp, as the following
examples show.  For the first inequality, take $G = \ZZ, G' = \{0\}$ and, of course, $\phi:
\ZZ \to \{0\}$ the trivial morphism.  Set $A = \k[t]$ with the obvious $\ZZ$-grading.
We see that $\id^\ZZ_A A = \id_A A = 1$, thus in this case, for $M=A$, the first
inequality is in fact an equality. On the other hand, $\id^\ZZ_A \k[t,t^{-1}] = 0$, but
$\id_A \k[t,t^{-1}] = 1$, so in this the case the second inequality is an equality.
Incidentally, the case where $G = \ZZ$ and $G' = \{0\}$ was already studied by E. Ekstr\"om,
see \cite{Ek}*{Theorem 0.2}.

Levasseur has proved in \cite{Lev}*{3.3} that if $A$ is noetherian and $\NN$-graded, its injective dimension and its graded
injective dimension are equal. The proof of this result uses a
spectral sequence which is not available if the grading group is not $\ZZ$. It would be
interesting to find a different proof using the change of grading functors, but we have
been unable to do so. We can, however, prove the following result, which holds even if the
algebra $A$ is not noetherian and thus provides a generalization of Levasseur's result.
Recall that an $\NN^n$-graded algebra is a $\ZZ^n$-graded algebra such that the support
$\{ \xi \in \ZZ^n \mid A_\xi \neq 0\}$ is contained in $\NN^n$.

\begin{proposition} 
Suppose $A$ is an $\NN^n$-graded algebra and let $\phi: \ZZ^n \to \ZZ$ be the morphism
defined by $\phi(a_1, \ldots, a_n) = a_1 + \cdots + a_n$. The injective dimension of $A$
in $\Mod^{\ZZ^n}_A$ is equal to its injective dimension in $\Mod^\ZZ_A$.  
\end{proposition} 
\begin{proof} 
Since we already know that $\id^{\ZZ^n}_A A \leq \id_A^\ZZ A$ by Proposition \ref{C:inj-dim},
we only need to prove the opposite inequality. For every $n \in \NN$ the set
$\phi^{-1}(n) \cap \NN^n$ is finite, so
\begin{align*} 
	\bigoplus_{\xi \in \phi^{-1}(n)} A_\xi = \prod_{\xi \in \phi^{-1}(n)} A_\xi.  
\end{align*} 
It follows from this that $\phi_!(A) = \phi_*(A)$ and, since $\phi_*$ is right adjoint to the exact functor $\phi^*$, it
preserves injectives; in particular, $\id^\ZZ_A \phi_*(A) = \id^\ZZ_A \phi_!(A)$
is at most $\id^{\ZZ^n}_A A$, which completes the proof.
\end{proof}


\end{document}